\documentclass{amsart}
\usepackage{amsfonts}
\usepackage{mathrsfs}
\usepackage{graphicx}

\newtheorem{theorem}{Theorem}[section]
\newtheorem{lemma}[theorem]{Lemma}
\newtheorem{corollary}[theorem]{Corollary}
\newtheorem{proposition}[theorem]{Proposition}

\theoremstyle{definition}

\theoremstyle{remark}
\newtheorem{remark}[theorem]{Remark}

\numberwithin{equation}{section}

\begin{document}

\title{Non-rigidity of spherical inversive distance circle packings}
\author{Jiming Ma}

\address{School of Mathematical Sciences, Fudan University, Shanghai,
China,
  200433}

\email{majiming@fudan.edu.cn}

\author{Jean-Marc Schlenker}
\address{
Institut de Math\'ematiques de Toulouse, UMR CNRS 5219, Universit\'e Paul Sabatier, 118 route de Narbonne, 31062 Toulouse Cedex 9, France
} \email{schlenker@math.univ-toulouse.fr}

\date{May 3, 2011}

\thanks{The first author was supported in part by NSFC 10901038 and Shanghai NSF
10ZR1403600. The second author was supported in part by ANR programs ETTT, ANR-09-BLAN-0116-01, 2009-13,
and ACG, ANR-10-BLAN-0105, 2010-14.}


\begin{abstract}
We give a counterexample  of Bowers-Stephenson's conjecture in the spherical case:
spherical  inversive distance circle packings are not determined  by their inversive distances.
\end{abstract}

\maketitle

\section{Introduction}

In this note we study inversive distance circle packing metrics on a surface $F$.

\subsection{Polyhedral surface}

Given  a  triangulated closed orientable surface $F$, a Euclidean (resp. spherical or hyperbolic) \emph{polyhedral surface}
is a map $l: E \rightarrow \mathbb{R}^{+}$, where $E$ is the set of all edges of the triangulation,
such that when  $e_{1}$, $e_{2}$ and  $e_{3}$ are  the three edges of a triangle, then $l(e_{1})+l(e_{2})>l(e_{3})$
(it is also  required  that $l(e_{1})+l(e_{2})+l(e_{3})< 2\pi$ in the spherical case).
From this $l$, there is a \emph{polyhedral metric} in $F$ such that the restriction of the metric
to each triangle is isometric to a triangle in $\mathbb{E}^{2}$ (resp. $\mathbb{S}^{2}$ or $\mathbb{H}^{2}$)
and  the length of an edge $e$ is given by $l(e)$.
For instance,  the boundary of a generic convex polyhedron in $\mathbb{E}^{3}$ (resp. $\mathbb{S}^{3}$
or $\mathbb{H}^{3}$) admits a natural polyhedral metric.

The \emph{discrete curvature} $k$ of a polyhedral surface is the map $k:V \rightarrow \mathbb{R}$,
where  $V$ is the set of all vertices of the triangulation, and for a vertex $v \in V$,
$k(v)=2\pi-\sum^{m}_{i=1}\theta_{i}$, where $\theta_{i}$ are the angles at the vertex $v$.

\subsection{Inversive distance circle packings}

The notion of inversive distance circle packing was introduced by Bowers-Stephenson in \cite{bw04},
it is a generalization  of Andreev and Thurston's circle packings on a surface, where two circles may intersect or not.
We just give the definition of the spherical inversive distance circle packing,
for Euclidean and hyperbolic cases, see \cite{bw04} and \cite{l10} for more detailed discussions.

For two circles $\mathcal{C}_{1}$  and  $\mathcal{C}_{2}$  centered at $v_{1}$, $v_{2}$ of radii $r_{1}$ and $r_{2}$ in $\mathbb{S}^{2}$,  so that $v_{1}$ and $v_{2}$ are of
distance $l$ apart, the \emph{inversive distance} $I=I(\mathcal{C}_{1}, \mathcal{C}_{2})$ between them is
\begin{equation}
I =\frac{\cos(l)- \cos(r_{1}) \cos(r_{2})}{\sin(r_{1}) \sin(r_{2})}.
\end{equation}

When viewed in $B^{3}$ considered as the Klein model of $\mathbb{H}^{3}$, the inversive distance is essentially the hyperbolic
distance (or the intersection angle) between the two totally geodesic planes in $\mathbb{H}^{3}$ with $\mathcal{C}_{i}$ as their ideal boundaries.
When those planes intersect, the inversive distance is the $\cos$ of their angle, and if they're disjoint,
it is the $\cosh$ of their distance.

In a triangulated surface $F$, a \emph{spherical inversive distance circle packing} is given as follows: fix a vector $I \in [-1, \infty)^{E}$, called the \emph{inversive distance vector}. For any $r \in (0, \infty)^{V}$, called the \emph{radius vector}, define the edge length by $l(e)=\sqrt{r(u)^{2}+r(v)^{2}+2r(u)r(v)I(e)}$ for an edge $e$ with $u$ and $v$ as its end points.
If for any triangle with $e_{1}$,  $e_{2}$ and  $e_{3}$ as its three edges, we have $l(e_{1})+l(e_{2})>l(e_{3})$ and $l(e_{1})+l(e_{2})+l(e_{3})< 2\pi$, then the edge length function $l: E \rightarrow \mathbb{R}$ defines a spherical polyhedral metric on $F$, which is called the \emph{spherical inversive distance circle packing  metric} with inversive distance $I$.

The geometric meaning is that in  $F$ with this polyhedral metric, if we draw circles with radii $r$ at the vertices $V$, then the inversive distance of two circles at the end points  of an edge $e$ is the given number $I(e)$.

It was conjectured by Bowers and Stephenson \cite{bw04} that inversive distance circle packings have a global rigidity property:
an inversive distance circle packing is determined by its combinatoric, inversive distance vector and discrete curvature at
the vertices.
Luo \cite{l10} proved Bowers-Stephenson's conjecture in the hyperbolic and Euclidean cases.
In this note, we give a counterexample in the spherical case:

\vskip 3mm
{\bf  Theorem 2.4.}
 \emph{There is a triangulation of $S^{2}$ and two spherical inversive distance circle packings with the same
inversive distance and discrete curvature, but they are not M\"{o}bius  equivalent.}
\vskip 3mm

The example we construct actually have zero discrete curvature at all vertices, so they are inversive distance
circle patterns on the (non-singular) sphere.

{\bf Acknowledgements}: This work was done when the first author was visiting Institut de  Math\'{e}matiques de Toulouse,
Universit\'{e} Paul Sabatier (Toulouse III), he would like to thank it for its hospitality.
He also would like to thank the China Scholarship Council for financial support.

\section{Proof of the theorem}

The proof of our theorem uses a well-known infinitesimal flexible Euclidean polyhedron
and the Pogorelov map which  preserves the relative distances between two points in the configurations in different geometries.
We first give a rapid preliminary.

\subsection{The hyperbolic and the de Sitter space}

Let $\langle x,y \rangle= -x_{0}y_{0}+x_{1}y_{1}+x_{2}y_{2}+x_{3}y_{3}$ be the symmetric 2-form  in the Minskowski space $\mathbb{R}^{4}_{1}$, recall that
the hyperbolic space is
\begin{equation}
\mathbb{H}^{3}=\{ x \in \mathbb{R}^{4}_{1} | ~\|x\|^{2}=-1, x_{0}> 0\}
\end{equation}
with the induced Riemannian metric on it,
which is a hyperboloid in $\mathbb{R}^{4}_{1}$.
The totally geodesic planes in $\mathbb{H}^{3}$ are the intersections between $\mathbb{H}^{3}$ and hyperplanes in $\mathbb{R}^{4}$ which pass through the origin.

Let $B^{3}$ be the unit ball in $\mathbb{R}^{3}$, then, there is a projective map $p_{\mathbb{H}}: \mathbb{H}^{3}\rightarrow B^{3}$ given by $\rho((x_{0}, x_{1}, x_{2}, x_{3}))=(x_{1}, x_{2}, x_{3})/x_{0}$, which is a homeomorphism and which maps geodesics in $\mathbb{H}^{3}$ into geodesics in $\mathbb{R}^{3}$.
This map is the projective model (Klein model) of the hyperbolic space.

A hyperideal hyperbolic polyhedron is the image of $p_{\mathbb{H}}^{-1}:  Q\cap B^{3}\rightarrow\mathbb{H}^{3}$,
where $Q$ is a Euclidean polyhedron in $\mathbb{R}^{3}$ such that all vertices of $Q$ lie out of $B^{3}$ and all edges of $Q$ intersect with $B^{3}$.
For a point $A$ in $\mathbb{R}^{3}- \overline{B^{3}}$, consider the space $A^{\bot}$ of
the points in $\mathbb{R}^{4}_{1}$ which are orthogonal to $p_{\mathbb{H}}^{-1}(A)$ in the symmetric 2-form.
$p_{\mathbb{H}}^{-1}(A)$ is a hyperbolic plane in $\mathbb{H}^{3}$. Then take  $A^{*}=p_{\mathbb{H}}(A^{\bot}) \cap B^{3}$,
so it is a hyperbolic plane in the Klein model of the 3-dimensional hyperbolic space $B^{3}$.
Thus its boundary is a round circle $\mathcal{C}_{A}$ in $\partial B^{3}=S^{2}$. $A^{*}$ is called the hyperbolic plane dual to $A$.
By a simple argument, the planes dual to two vertices  $A$ and $B$ of $Q$ don't intersect (this
is because the segment with endpoints $A$ and $B$ intersects the ball).
The length of an edge  of a hyperideal hyperbolic polyhedron is defined as the distance between the dual planes.

The de Sitter space can be defined as
\begin{equation}
\mathbb{S}^{3}_{1}=\{ x \in \mathbb{R}^{4}_{1} | ~\|x\|^{2}=1\}
\end{equation}
with the induced Lorentzian metric on it,
which is a one-sheeted hyperboloid in $\mathbb{R}^{4}_{1}$.
The totally geodesic planes in $\mathbb{S}^{3}_{1}$ are the intersections of $\mathbb{S}^{3}_{1}$
with the hyperplanes in $\mathbb{R}^{4}$ which pass through the origin.
Let
\begin{equation}
\mathbb{S}^{3}_{1, +}=\{ x \in \mathbb{R}^{4}_{1} ~ |  ~  \|x\|^{2}=1, x_{0}> 0\}
\end{equation}
be the upper de Sitter space.

As for the hyperbolic space, there is a projective map
$p: \mathbb{S}^{3}_{1, +}\rightarrow \mathbb{R}^{3}-\overline{B^{3}}$ given by
$\rho((x_{0}, x_{1}, x_{2}, x_{3}))=(x_{1}, x_{2}, x_{3})/x_{0}$, which is a homeomorphism and which
maps geodesic in $\mathbb{S}^{3}_{1, +}$ into geodesic in $\mathbb{R}^{3}$.

In the projective model of $\mathbb{S}^{3}_{1, +}$,
a geodesic maybe pass through $B^{3}$, and if it is the case, then the geodesic is \emph{time-like}. If a geodesic does not pass through the closure of $B^{3}$, then this geodesic is \emph{space-like}.

For more details on distances in the  de Sitter  space,  see \cite{s98}: for two points $x$ and $y$ in $\mathbb{S}^{3}_{1, +}$, if the geodesic  $[x,y]$ is a time-like geodesic, then the \emph{distance} $d$ between them is the negative number $d$ such that $\cosh(d)=\langle x, y\rangle$; if the geodesic  $[x,y]$ is a space-like geodesic, the \emph{distance} $d$ between them is  the unique number in $i[0,\pi]$ such that $\cosh(d)=\langle x, y\rangle$.

There is a duality between points in the de Sitter  space and oriented hyperplanes in the hyperbolic 3-space: consider the projective model of the upper de Sitter  space $\mathbb{R}^{3}- \overline{B^{3}}$, when $A$ lies in $\mathbb{R}^{3}- \overline{B^{3}}$, then the hyperplane  $A^{*}$ constructed above viewed as in hyperbolic 3-space is the dual of $A$.

When $A$ and $B$ are  two points
in $\mathbb{R}^{3}- \overline{B^{3}}$, such that the Euclidean line $L$ connecting $A$ to $B$ passes
through  $B^{3}$, then the de Sitter distance between $A$ and $B$  is essentially  the hyperbolic distance
between the two planes $A^{*}$ and $B^{*}$: let $l$ be the distance between $A^{*}$ and $B^{*}$ in the hyperbolic space, then  $l=-d$. It is also essentially the inversive distance between the two circles
$\mathcal{C}_{A}$ and $\mathcal{C}_{B}$, where $\mathcal{C}_{A}$ and $\mathcal{C}_{B}$ are the ideal boundaries of the planes $A^{*}$¡¡and $B^{*}$ in $S^{2}= \partial B_{3}$.


\subsection{A flexible polyhedron}

We now describe a well-known example of an infinitesimally flexible polyhedron, which will be
the keystone of the construction of the counter-example presented here.

\begin{lemma}  There is a  Euclidean polyhedron $Q$ such that
\begin{enumerate}
\item all vertices of $Q$ lie out of $B^{3}$,
\item all edges of $Q$ intersect with $B^{3}$,
\item in any neighborhood of $Q$, there are two Euclidean  polyhedra $Q_{t}$ and $Q_{-t}$ which have the same combinatorics as $Q$ and the same corresponding edge lengths.
\end{enumerate}
\end{lemma}

\begin{proof}
We first recall Sch\"{o}nhardt's twisted octahedron (see \cite{s28} and \cite{i10}): let $ABC$ be an
equilateral triangle in $\mathbb{R}^{3}$, and let $L$ be a line that passes through the center
of $ABC$ and it is orthogonal to the plane of the triangle. Let $A^{0}B^{0}C^{0}$ be the image
of $ABC$ under a screw motion with axis $L$ and rotation angle $\pi/2$. Consider
a polyhedron $Q$ bounded by triangles $ABC$, $A^{0}B^{0}C^{0}$, $ABC^{0}$, $ A^{0}BC$, $ AB^{0}C$,
$A^{0}B^{0}C$, $AB^{0}C^{0}$, and $ A^{0}BC^{0}$. The polyhedron $Q$ is combinatorially isomorphic
to an octahedron, and has three edges with dihedral angles bigger than $\pi$:
the edges $AB^{0}$, $BC^{0}$, and $CA^{0}$, see Figure 1.

\begin{center}
\scalebox{0.50}[0.50]{\includegraphics {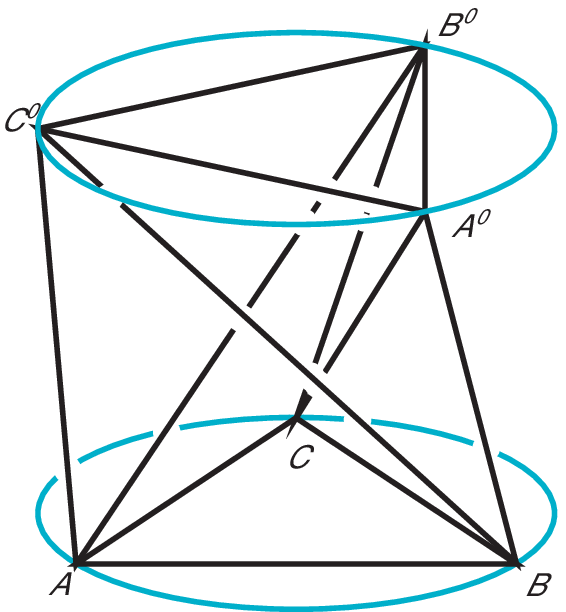}}
\end{center}

\begin{center}
Figure 1. Sch\"{o}nhardt's twisted octahedron
\end{center}

This polyhedron $Q$ is infinitesimal flexible (see \cite{i10}): there are vectors
$\eta_{A^{0}}$, $\eta_{B^{0}}$, $\eta_{C^{0}}$ in $\mathbb{R}^{3}$ such that the polyhedron $Q_{t}$
with the vertices $A$, $B$, $C$, $A^{0} + t \eta_{A^{0}}$, $B^{0} + t \eta_{B^{0}}$, $C^{0} + t \eta_{C^{0}}$,
which is a small deformation of  $Q$, is a non-trivial infinitesimally isometric deformation of $Q$,
where $\eta_{A^{0}}$ is a vector orthogonal to the plane $A^{0}BC$ of norm 1 and pointing out from $Q$,
and similarly for $\eta_{B^{0}}$ and $\eta_{C^{0}}$.
Then by  a direct calculation(or see Lemma 4.1 of \cite{i10}), the pairs of corresponding edges of $Q_{t}$ and $Q_{-t}$
have the same lengths, for $0<t$ small enough.

Let $O$ be the center of the polyhedron $Q$, i.e, $O$ lies in the line $L$ and its distances to the planes
$ABC$ and $A^{0}B^{0}C^{0}$ are both equal to $h>0$. Let $a$ be the edge length of the equilateral triangle $ABC$, a simple calculation shows
that conditions (1) and (2) of the lemma are  equivalent to
\begin{enumerate}
\item $h^{2}+a^{2}/3> 1$,
\item $h^{2}+a^{2}/12< 1$,
\item  $a^{2}/3< 1$.
\end{enumerate}

So, we can assign $h=1/2$ and $a$ a little bigger than $3/2$, and the lemma follows.
\end{proof}

Pogorelov  \cite{p73} has found  remarkable maps from $\mathbb{S}^{3}_+ \times \mathbb{S}^{3}_+$ and
$\mathbb{H}^{3} \times \mathbb{H}^{3}$ to $ \mathbb{R}^{3} \times \mathbb{R}^{3}$, see also \cite{s98}, \cite{s00}, \cite{s05}
for other forms of these maps and their infinitesimal versions.
What we really need is the first four properties of the following proposition, we state it along with other properties for the sake of future reference.

\begin{proposition}
There exists a map $\Phi:  \mathbb{S}_{1,+}^{3} \times \mathbb{S}_{1,+}^{3} \rightarrow \mathbb{R}^{3} \times \mathbb{R}^{3}$ such that:
\begin{enumerate}
\item  $\Phi$ is a homeomorphism from $\mathbb{S}_{1,+}^{3} \times \mathbb{S}_{1,+}^{3}$ to its image in $\mathbb{R}^{3} \times \mathbb{R}^{3}$,
\item
the restriction of $\Phi$ to the diagonal $\triangle \subset \mathbb{S}_{1,+}^{3} \times \mathbb{S}_{1,+}^{3}$ is the projective map $p$ 
(its image is in the diagonal  $\triangle' \subset \mathbb{R}^{3} \times \mathbb{R}^{3})$,
 \item  let $\alpha$ be a time-orientation preserving isometry of $\mathbb{S}_{1}^{3}$. There is then a Euclidean
isometry $\beta$ as follows. Let $x\in \mathbb{S}_{1,+}^{3}$ with $\alpha(x)\in \mathbb{S}_{1,+}^{3}$, 
we have $\Phi(x, \alpha(x))=(y, y')$ in $\mathbb{R}^{3}$, and we have $y'=\beta(y)$,
\item if $[x,y]$ and $[x',y']$ are two time-like geodesics of the same length in $\mathbb{S}_{1,+}^{3}$,
and, if $p_{1}, p_{2}$ are the projections of $\mathbb{R}^{3} \times \mathbb{R}^{3}$ on the two factors,
then $[p_{1}\circ \Phi \circ (x,x'),p_{1}\circ \Phi \circ (y,y')]$ and $[p_{2}\circ \Phi \circ (x,x'),p_{2}\circ \Phi \circ (y,y')]$
are geodesics of the same length in $\mathbb{R}^{3}$,
\item
if $g_{1},g_{2}:[0,1]\rightarrow \mathbb{S}_{1,+}^{3}$ are space-like geodesic segments parametrized at the same speed,
 then $p_{1}\circ \Phi \circ (g_{1},g_{2})$ and  $p_{2}\circ \Phi \circ (g_{1},g_{2})$ are geodesic segments parametrized at the same speed,
\item there exists a point $x^{0}= p^{-1}(0) \in \mathbb{S}_{1,+}^{3}$ such that, for each 2-plane $\Pi \subset \mathbb{S}_{1,+}^{3}$
containing $x^{0}$,
\begin{equation}
\forall x \in \Pi, ~\forall y \in \mathbb{S}_{1,+}^{3},~p_{1}\circ \Phi(x,y) \in p(\Pi).
\end{equation}
\end {enumerate}
\end{proposition}

The proof of this proposition can be obtained by following those given by Pogorelov's book
for the hyperbolic space or the sphere. More precisely, it is straightforward to adapt the proof of
\S3 Lemmas 1-4 and \S4 Theorems 1-2 in  Chapter V of \cite{p73}.

Or from Section 6 of \cite{v10}: in Proposition 6.3 and 6.4 of \cite{v10}, we should replace $f(a,b)= (a^{2}-b^{2})^{2}-8(a^{2}+b^{2}-2)$
(in the hyperbolic case and $1>a,b \geq 0$) to $g(a,b)= -(a^{2}-b^{2})^{2}+8(a^{2}+b^{2}-2)$ (in de Sitter case and $1< a,b$). Note that $g(a,b)= -(a^{2}-b^{2})^{2}+8(a^{2}+b^{2}-2)$ is not always positive for $1< a,b$, but this is true for $-4<a^{2}-b^{2}<4$, so $\{(\xi,\eta) \in (\mathbb{R}^{3}-B^{3})\times (\mathbb{R}^{3}-B^{3})| -4<|\xi|^{2}-|\eta|^{2}<4\} \subset Im(\Phi)$, which is an open  neighborhood of the diagonal of $(\mathbb{R}^{3}-B^{3})\times (\mathbb{R}^{3}-B^{3}) \subset \mathbb{R}^{3}\times \mathbb{R}^{3}$.

For the proof of (3) of Proposition 2.2, we just recall that for a time-orientation preserving isometry  $\alpha$ of $\mathbb{S}_{1}^{3}$, in the matrix presentation $A_{4\times 4}$ of it, the $(1,1)$-entry of $A$ is positive, and then (3) of Proposition 2.2 follows from arguments similar to Proposition 6.5 of \cite{v10}.

For the proof of (4) of Proposition 2.2, we need the transitivity of the time-orientation subgroup of $Iso(\mathbb{S}_{1}^{3})$ on the space of time-like geodesic segments of a fixed length, which can be see from the duality between the de sitter space and the hyperbolic space. From this, we have an  time-orientation isometry $\alpha$, such that $\alpha([x,y])=[x',y']$, and then (4) follows from (3). (5) is similar.

The map $\Phi$ in the Minkowski coordinate is given as follows: let  $(x,y)\in \mathbb{S}_{1,+}^{3} \times \mathbb{S}_{1,+}^{3}$, where
$x=(x_{0}, x_{1}, x_{2}, x_{3})$ and $y=(y_{0}, y_{1}, y_{2}, y_{3})$, then
\begin{equation}
\Phi(x,y)=2((x_{1}, x_{2}, x_{3}),(y_{1}, y_{2}, y_{3}))/(x_{0}+y_{0})\in \mathbb{R}^{3} \times \mathbb{R}^{3}.
\end{equation}

The converse  $\Phi^{-1}:\mathbb{R}^{3} \times \mathbb{R}^{3} \rightarrow \mathbb{S}_{1,+}^{3} \times \mathbb{S}_{1,+}^{3}$ is given by
\begin{equation}
((\xi_{1}, \xi_{2}, \xi_{3}),(\eta_{1}, \eta_{2}, \eta_{3}))\rightarrow (\rho((4-|\eta|^{2}+|\xi|^{2},\xi_{1}, \xi_{2}, \xi_{3})),\rho((4-|\xi|^{2}+|\eta|^{2},\eta_{1}, \eta_{2}, \eta_{3}))),
\end{equation} where $\rho$ is the linear normalization such that $\rho((4-|\eta|^{2}+|\xi|^{2},\xi_{1}, \xi_{2}, \xi_{3}))$ and $\rho((4-|\xi|^{2}+|\eta|^{2},\eta_{1}, \eta_{2}, \eta_{3}))$ lie in the hyperboloid $\mathbb{S}_{1,+}^{3}$.

\begin{remark} The Pogorelov maps in  \cite{p73} and \cite{v10} are a little different, 
i.e, they are equal up to  the multiple constant 2, we choose the one similar to \cite{v10}, which is convenient for us.
\end{remark}

\begin{theorem}
There is a triangulation of $S^{2}$ and two spherical inversive distance circle packings
with the same inversive distance and  discrete curvature, but they are not M\"{o}bius equivalent.
\end{theorem}

\begin{proof}  The triangulation of $S^{2}$ is given from the boundary of the Euclidean octahedron in Lemma 2.1. From Lemma 2.1,  we have two Euclidean polyhedra which have the same edge lengths, but they are not congruent, say $Q_{t}$ and $Q_{-t}$ for a fixed  $t>0$ small enough, which are very near to $Q$. Denoted the vertices of $Q$ by $v^{i}$ and the corresponding vertices of $Q_{t}$ (resp. $Q_{-t}$) by $v^{i}_{t}$ (resp. $v^{i}_{-t}$). Note that $(v^{i},v^{i}) \in Im(\Phi)$  by Proposition 2.2 (2), and from Proposition 2.2 (1), we can assume  $(v^{i}_{t},v^{i}_{-t}) \in Im(\Phi)$. Moreover $\Phi^{-1}\circ (v^{i}, v^{i})$ gives us  a polyhedron in $\mathbb{S}_{1,+}^{3}$ such that each of the edges are time-like. Then  we use Proposition 2.2 (4),  $p_{1}\circ \Phi^{-1}\circ (v^{i}_{t}, v^{i}_{-t})$ and $p_{2}\circ \Phi^{-1}\circ (v^{i}_{t}, v^{i}_{-t})$ give us two polyhedra, say $P_{t}$ and $P_{-t}$,  in $\mathbb{S}_{1,+}^{3}$ such that each of the edges are time-like and which have the same corresponding edge lengths. For each vertex of $P_{t}$, we have a circle in $S^{2}$, which is the ideal boundary of the hyperbolic plane dual to the vertex.  But recall that the de Sitter length here is essentially the inversive  distance of the circles corresponding to two ideal vertices of the  hyperideal hyperbolic polyhedra. So, we have  two spherical inversive distance circle packing metrics, they induced the same standard spherical metric in $\mathbb{S}^{2}$, thus they have the same discrete curvature zero. These  two spherical inversive distance circle packing are not M\"{o}bius equivalent, as can be seen also from the Pogorelov map.

\end{proof}

\begin{corollary}  There is a hyperideal polyhedron $P$ such that each face of it is a triangle and in any neighborhood of $P$, there are two hyperideal polyhedra $P_{t}$ and $P_{-t}$ which have the same combinatorics and the corresponding edges of them have the same length.
\end{corollary}

\begin{proof} Now from Lemma 2.1, we have two Euclidean polyhedra $Q_{t}$ and $Q_{-t}$, which have the same edge length, but they are not  congruent. Then we use Proposition 2.2, we get two polyhedra in $\mathbb{S}_{1,+}^{3}$ such that each of the edges are time-like.

Such polyhedra  in $\mathbb{S}_{1,+}^{3}$ can be viewed as hyperideal hyperbolic polyhedra, and the distance in the de Sitter geometry is just  the distance of the circles corresponding to two ideal vertices of the hyperideal hyperbolic polyhedra, which is the edge length of the hyperideal hyperbolic polyhedra.
\end{proof}

\begin{remark}
Our polyhedra above are not convex,  a similar phenomena appears in convex hyperbolic  polyhedra, but, some of the faces are not triangle, see Theorem 2' of \cite{s00}.

\end{remark}

\bibliographystyle{amsplain}

\end{document}